\documentclass[11pt,reqno]{amsart}

\RequirePackage[OT1]{fontenc}
\RequirePackage{amsthm,amsmath}
\RequirePackage[numbers]{natbib}

\usepackage{tikz}

\usepackage{amsbsy}
\usepackage{amsmath}
\usepackage{amsthm}
\usepackage{latexsym}
\usepackage{amssymb}
\usepackage{amsfonts}
\usepackage{epsfig}
\usepackage{stackrel}
\usepackage{accents}

\usepackage{lineno}
\usepackage[normalem]{ulem}
\usepackage[backref=section,colorlinks=true,urlcolor=blue,citecolor=blue,pdfstartview=FitH]{hyperref}

\usepackage{psfrag}

\newtheorem{theorem}{Theorem}
\newtheorem{proposition}{Proposition}
\newtheorem{lemma}{Lemma}

\makeatletter
\newcommand*\bigcdot{\mathpalette\bigcdot@{.5}}
\newcommand*\bigcdot@[2]{\mathbin{\vcenter{\hbox{\scalebox{#2}{$\m@th#1\bullet$}}}}}
\makeatother

\theoremstyle{remark}
\newtheorem{remark}{Remark}

\DeclareMathOperator{\sgn}{sgn}

\newcommand{\weak}{\Rightarrow}
\newcommand{\drm}{\mathrm d}

\def\sM{\mathtt{M}}

\def\sF{\mathtt{F}}

\def\sL{\mathtt{L}}
\def\sl{\mathtt{l}}

\def\expect{{\mathbb  E}}
\def\Pr{{\mathbb P}}
\def\real{{\mathbb  R}}
\def\nat{{\mathbb  N}}
\def\integer{{\mathbb Z}}

\newcommand\ind[1]{{1}_{\{#1\}}}

\def\invexcl{\rotatebox[origin=c]{180}{!}}

\newcommand{\var}[0]{\operatorname{Var}}
\newcommand{\floor}[1]{\lfloor #1 \rfloor}
\newcommand{\ceil}[1]{\lceil #1 \rceil}
\newcommand{\mymax}[1]{#1_\vee}

\usepackage{psfrag}
\usepackage{epsfig}

\begin{document}

\title{Extreme Values in Closed Networks}

\author{Predrag Jelenkovi\'c}
\address{Department of Electrical Engineering, Columbia University, New York, NY 10025}
\email{predrag@ee.columbia.edu}

\author{Petar Mom\v cilovi\'c}
\address{Department of Industrial and Systems Engineering, Texas A\&M University, College Station, TX 77843}
\email{petar@tamu.edu}

\begin{abstract}
For a widely used hub-and-spoke closed product-form network consisting of an infinite-server node and several single-server queues, we characterize the maximum queue-length distribution in various operational regimes by leveraging a novel probabilistic representation of the joint queue-length distribution and scaling where the number of customers grows. In these limiting regimes, we derive explicit characterizations of the maximum that are asymptotically equivalent to the maximum of independent random variables with the same geometric marginal distribution as queue lengths. 
In particular, when both the number of customers and queues grow, the parameters of the marginal distribution depend on the global characteristics of the network and are explicitly computed from a quadratic equation that arises from the corresponding large-deviation rate functions. Explicit computation of global characteristics of product-form distribution beyond the marginals, e.g., the maximum, appears novel, and our methodology may apply to other global measures of interest. 
\end{abstract}

\keywords{Closed network, product-form solution, extreme values}

\date{\today}

\maketitle

\section{Introduction}

Closed queueing networks have been used to model systems such as communication and computer networks~\cite{Pit79,MMi82,BBK99}, scrip economies~\cite{JSS14}, and shared mobility services such as ride-sharing and car rental~\cite{BFL22,BDL19,BLW22}; see~\cite{ADW13} and references therein. Recently, these models have also been applied in the context of biomolecular networks~\cite{MLH17}, where monomers (customers) stochastically attach to and detach from filaments (single-server queues) from a free monomer pool (infinite-server node) as part of reversible chemical processes; hence, time-reversible Markov chains are natural model choices. 

The motivation for our results arises from recent molecular cable models in single-cell organisms, where multiple filaments (queues) are grouped into a cable, and the distribution of the cable length can be described as the maximum of many single-server queues~\cite{RMJ23,MRR24}. 
In addition to the plausibility arguments, these studies demonstrated experimentally that the mean and variance of the measured cable lengths scale in accordance with the extreme-value Gumbel distribution. 
The model used in ~\cite{RMJ23,MRR24} assumed that monomers attach to filaments from an infinite monomer pool at constant rates $\lambda$ and detach at rates~$\mu$, representing an open queueing network in which individual single-server queues are independent and geometric, with a parameter determined by the infinite queue utilization $\rho=\lambda/\mu$. However, as previously noted, it is well established that cellular filaments and cables are assembled from a finite monomer pool~\cite{MLH17}. This motivates our investigation of the distribution of the maximum and its scaling in a closed queueing network.
Moreover, understanding the maximum queue-length distribution is of general interest for properly sizing service networks, where efficiency and service quality are balanced, since the maximum queue size and delay can substantially exceed those of individual queues. 

From a mathematical perspective, this work complements our previous work in~\cite{JKM22} that focused on identifying different operational regimes of the network in the context of computing the marginal distribution of the single-server queues. 
Although the considered network admits a stationary product-form solution (see~\eqref{eq:prodpi}), this solution provides little qualitative insight into the spectrum of feasible operating regimes. This motivated our probabilistic approach for explicitly evaluating the marginal distribution of individual queues, bypassing the problem of computing the normalization constant (partition function). Indeed, the presence of the normalization constant in~\eqref{eq:prodpi} has a profound impact on the behavior of the marginal distribution: the marginal can be approximately uniform in one regime and geometric in other regimes. Interestingly, in some operating regimes, the distribution of marginal queue lengths is determined by a parameter that can be obtained from an equation that involves large-deviation rate functions.
The explicit nature of our solutions allowed us to precisely characterize different operating regimes, providing new insights into network performance. 

In this paper, we extend the methodology of~\cite{JKM22} to derive an explicit characterization of the distribution of the maximum queue length. 
Moreover, our analysis further indicates that the global structure of the network distribution is asymptotically equivalent to a uniform distribution on a simplex, whose size is determined by the network parameters. In particular, we demonstrate the asymptotic equivalence of the maximum queue to a maximum on a simplex, which further reveals the dependence between queues through known results from the extreme-value theory for independent random variables.
More formally, Theorem~\ref{thm:disc} in Section~\ref{sec:main} establishes that the maximum distribution can be asymptotically equivalent to the maximum of independent random variables sharing the same geometric marginal distribution as the queue lengths. The parameter of this geometric distribution depends on the global characteristics of the network. In the case of Theorem~\ref{thm:disc}, when both the number of customers and queues grow proportionally, the parameter of the geometric distribution is determined as the solution to a quadratic equation arising from the corresponding large-deviation rate functions. Furthermore, the total number of customers across all single-server queues is approximately constant in this regime, implying that the joint distribution of the queues is approximately uniform on a discrete simplex. 
Conversely, when the number of queues is fixed as the number of customers grows, the maximum queue converges in distribution to the maximum on a continuous simplex adjusted for the expected number of customers in the infinite server queue, see Theorem~\ref{thm:cont}. In Section~\ref{sec:nonhomogen}, we discuss how our method can be extended to the case of a non-homogeneous network with different utilization levels of single-server queues, as stated in Theorems~\ref{thm:3} and~\ref{thm:4}. We omit rigorous proofs of these two results due to tedious computations. 
Finally, beyond the analysis of the maximum, our methodology can be applied to computing the distribution of other global network measures, such as percentiles and order statistics, as discussed in Section~\ref{sec:cr}.


The paper is organized as follows. In the next section, we describe our model and analysis approach. We state results on vectors uniformly distributed on simplexes in Section~\ref{sec:simplex}. Section~\ref{sec:main} contains our main results. We discuss the case of a non-homogeneous network in Section~\ref{sec:nonhomogen}. Concluding remarks and technical proofs can be found in Section~\ref{sec:cr} and Section~\ref{sec:proofs}, respectively.

\section{Model \& Approach}
\label{sec:model}

We consider a closed network model with $m$ single-class customers that consists of an infinite-server node and $n$ infinite-buffer single-server queues. The service times at the infinite-server node are exponentially distributed with rate~$\lambda$. Once a customer completes the service at this node, it is routed to the single-server queue~$i$ with probability~$p_i>0$, where $\sum_{i=1}^n p_i = 1$. At queue~$i$, service times are also exponentially distributed, but with rate~$\mu_i$; all service times across the queues and time are independent. After completing service at a single-server queue, a customer rejoins the infinite-server node. We consider the case of homogeneous queues: all single-server queues have the same utilization, which is proportional to $\kappa^{-1} = p_i \lambda/\mu_i$ for all~$i$; the non-homogeneous case is discussed in Section~\ref{sec:nonhomogen}. Let $\pi(l)$, $l=(l_{1},\ldots,l_{n}) \in \integer_0^n:=\{0,1,\ldots\}^n$, be the steady-state probability that lengths of queues $1,\ldots,n$ are equal to $l_{1},\ldots,l_{n}$, respectively. Then, it is well known that $\pi$ is product-form and satisfies (e.g., see~\cite{Kog95})
\begin{equation}
\pi(l) = c_{m,n}^{-1} \, m! \frac{ \kappa^{-|l|}}{(m- |l|)!}, \label{eq:prodpi}
\end{equation}
where $c_{m,n}^{-1} = \pi(0)$ is a normalization constant (partition function) and $|l|:= l_1+\cdots+l_n \leq m$. Let $L=(L_1,\ldots,L_n)$ be a random vector distributed according to $\pi$; set $M :=m-|L|$. In steady state, $L_i$ is the random length of the $i$th single-server queue, and $M$ is the random number of customers in the infinite-server node. The operational regime of such a network is determined by the relation between the number of customers~$m$, the number of queues~$n$, and the parameter~$\kappa$. As shown in~\cite{ADW13}, the set of bottleneck nodes plays a key role in the network behavior. Informally, the bottleneck nodes are idle with negligible probabilities, and they can be ``replaced" with Poisson sources, whereas the rest of the network can be viewed as a product-form open network. In our homogeneous setting, all queues either act as bottlenecks or none do. The latter occurs when the number of customers is too small to fully utilize the queues, making the customer pool a bottleneck resource.

Instead of analyzing equation~\eqref{eq:prodpi} directly or considering its transforms (as in~\cite{CLW95b,Kog95,KnT90,KBD21}), we utilize a probabilistic representation~\cite{MLH17} of the stationary distributions of various network quantities and concentration bounds. To this end, let~$\sM$ be a $\kappa$-mean Poisson random variable. Then, it was shown~\cite{JKM22} that $c_{m,n}$ satisfies
\begin{equation}
c_{m,n} = \frac{\expect[(m+n-1-\sM)_{n-1} \ind{\sM \leq m}]}{(n-1)! \,\Pr[\sM=m]}, \label{eq:cmn}
\end{equation}
where $\ind{\cdot}$ is the standard indicator function and $(n)_k$ denotes the falling factorial:
\[
(n)_k : = n(n-1) \cdots (n-k+1).
\]
The joint distribution of stationary single-server queue lengths can also be derived in terms of the Poisson variable~$\sM$. In particular, let $x\leq y$ denote $\{x_1 \leq y_1,\cdots,x_n \leq y_n\}$, for $x,y \in \real^n$. Then, \eqref{eq:prodpi} and~\eqref{eq:cmn} yield, for $l \in \real_0^n$,
\begin{align}
    \Pr[L \leq l] &= c_{m,n}^{-1} \frac{1}{\Pr[\sM=m]} \sum_{j=0}^m \sum_{x \in\integer_0^n: \, |x| = m-j } \ind{x \leq l} \Pr[\sM = j] \notag \\
    &= c_{m,n}^{-1} \frac{1}{\Pr[\sM=m]} \sum_{j=0}^m \sum_{x \in\integer_0^n: \, |x| = m-j } \frac{\binom{n+m-j-1}{n-1}}{\binom{n+m-j-1}{n-1}} \ind{x \leq l} \Pr[\sM = j] \notag \\
    &= \frac{\expect[(m+n-1-\sM)_{n-1} \ind{\sM \leq m} \ind{X(m-\sM) \leq l}]}{\expect[(m+n-1-\sM)_{n-1} \ind{\sM \leq m}]}, \label{eq:cdfprobrep}
\end{align}
where, conditional on $\{\sM=m-k\}$, $X(k) \in \integer_0^n:=\{0,1,\ldots\}$ is uniformly distributed on a simplex $\Delta^n_k :=\{x\in\integer_0^n:\, |x|=k\}$ of size~$k$. Random vectors uniformly distributed on simplexes play a key role in our analysis. The distribution of the maximum $\mymax{L} = \max_i L_i$ follows from the joint distribution:
\begin{equation}
    \Pr[\mymax{L} > k] 
    = \frac{\expect[(m+n-1-\sM)_{n-1} \ind{\sM \leq m} \ind{\mymax{X}(m-\sM) > k}]}{\expect[(m+n-1-\sM)_{n-1} \ind{\sM \leq m}]}, \label{eq:maxprobrep}
\end{equation}
where $\mymax{X}(m-\sM) := \max_i |X_i(m-\sM)|$ is the maximum element of $X(m-\sM)$.

Throughout the paper, we use the following notation. For a vector $x \in \real^d$, let $|x|=|x_1|+\cdots+|x_d|$ and $\mymax{x} = \max_i |x_i|$ denote the $\ell_1$ and $\ell_\infty$ norms, respectively. For $f,g: \, \real \to \real$, we write $f(x) \sim g(x)$ and $f(x) \gg g(x)$, as $x \to \infty$, to denote $f(x)/g(x) \to 1$ and $g(x)/f(x)\to 0$, as $x \to \infty$, respectively. Finally, define $H_k(n) := \sum_{i=1}^n i^{-k}$.

\section{Preliminary Results: Simplex}
\label{sec:simplex}

The probabilistic representations~\eqref{eq:cdfprobrep} and~\eqref{eq:maxprobrep} involve a random vector $X(m-\sM)$ uniformly distributed on a discrete simplex of random size $m-\sM$, where $\sM$ is a Poisson random variable with mean~$\kappa$. This section presents two results that characterize random vectors uniformly distributed on simplexes. Let $Y \in \integer_0^n$ be a vector with independent geometrically distributed components. Conditioned on the event $\{|Y|=m\}$, the vector~$Y$ is uniformly distributed on the discrete simplex $\Delta^n_m = \{l\in \integer_0^n:\, |l|=m\}$; this conditional distribution does not depend on the parameter of the geometric distribution. The following lemma estimates the probability that $|Y|$ takes a particular value in a (sub-CLT) neighborhood of $\expect |Y|$.

\begin{lemma} \label{lemma:sum}
Let $Y \in \integer^n_0$ be a vector with independent geometrically distributed components: $\Pr[Y_i=j]=(1-p)p^j$, $j \geq 0$, $i=1,\ldots,n$. Then, uniformly for all integer $|k| \ll \sqrt{n} \sigma$, 
    \[
    \sqrt{2\pi n \sigma^2} \, \Pr[|Y| = \floor{n \expect Y_1} + k ] \to 1, 
    \]
    as $\min\{n,n \expect Y_1\} \to \infty$, where $\expect Y_1 = p/(1-p)$ and $\sigma^2 = \var(Y_1) = p/(1-p)^2$.
\end{lemma}

\begin{proof} See Section~\ref{sec:prelimproof}.  
\end{proof}

The next proposition characterizes $\mymax{X}$ when $X \in \integer^n_0$ is uniformly distributed on a discrete simplex of size~$m$.

\begin{proposition} \label{pr:max_simplex_discrete}
Let $X \in \integer_0^n$ be a random vector uniformly distributed on a simplex of nonnegative integers 
$\Delta^n_m =\{l \in  \integer_0^n: \, |l|=m \}, m\in \nat$.
If $\ln n \gg \ln (1+n/m)$, as $n\to\infty$, then, as $n \to \infty$, 
\[
\frac{\ln(1+n/m)}{\ln n} \mymax{X}
\stackrel{\Pr}{\rightarrow} 1.
\]
\end{proposition}

\begin{remark}
    The condition $\ln n \gg \ln (1+n/m)$ prevents the dimension~$n$ from growing too rapidly relative to the size~$m$. For example, if $n \sim m^\alpha$ with $\alpha>1$, then $\frac{\alpha-1}{\alpha} \mymax{X}$ does not converge in probability to~$1$ because $\frac{\alpha}{\alpha-1}$ is not necessarily integer whereas~$\mymax{X}$ takes integer values only.
\end{remark}

\begin{proof} See Section~\ref{sec:prelimproof}.
\end{proof}

The last result in this section characterizes $\mymax{X}$ when $X \in \real^n_0:=[0,\infty)^n$ is uniform on a continuous unit simplex. In this case, $n \,\expect X_1 = 1$ and $n^2 \,\var(X_1) = \frac{n-1}{n+1}$. Hence, the proposition implies that the coefficient of variation of the maximum~$\mymax{X}$ is lower than that of marginal $X_1$. In particular, the coefficient of variation of~$\mymax{X}$ vanishes as~$n$ increases.

\begin{proposition} \label{pr:max_simplex_cont}
Let $X \in \real_0^n$ be a random vector uniformly distributed on a unit simplex of nonnegative reals 
$\Delta^n=\{x \in  \real_0^n: \, |x|=1 \}$. Then
\[
n\, \expect \mymax{X} = H_1(n) \quad\text{and}\quad n^2\,\var(\mymax{X}) = \frac{n H_2(n) - H^2_1(n)}{n+1}.
\]
Moreover, as $n\to\infty$,
\begin{equation}
\Pr[n\mymax{X} \leq x + \ln n ] \to e^{-e^{-x}}. \label{eq:simplexGumbel}
\end{equation}
\end{proposition}

\begin{proof} See Section~\ref{sec:prelimproof}. 
\end{proof}

\section{Main Results}
\label{sec:main}

This section presents the main results in Theorems~\ref{thm:disc} and~\ref{thm:cont}. The proofs of the theorems are based on~\eqref{eq:maxprobrep} and~\eqref{eq:cdfprobrep}, both of which involve the key term: 
\[
\sF_n(m-\sM) := (m+n-1-\sM)_{n-1} \ind{\sM \leq m}. 
\]
The first lemma identifies the values of~$(m-\sM)$ on which the mass in $\expect[\sF_n(m-\sM)]$ is concentrated; these values depend on the network parameters: $n$, $m$ and $\kappa$. Informally, the mass is concentrated on values of $(m-\sM)$ that are away from $(m-\expect \sM)$ in the first two operational regimes described by the lemma, corresponding to equation~\eqref{eq:con100}. This ``shift" $\zeta$ occurs because the falling factorial in $\sF_n(m-\sM)$ is exponentially large in~$n$, which compensates for the exponentially small probability of $\{\sM \leq m\}$.

\begin{lemma}[Concentration] \label{lemma:concentration}
If $\kappa/m \to \bar\kappa \geq 0$, as $m\to\infty$, then, for any $\epsilon>0$, as $m\to\infty$,
\begin{equation}
\frac{\expect[(m+n-1-\sM)_{n-1} \ind{\sM \leq m} \ind{ |(m-\sM)/n - \zeta)| \leq \epsilon} ]}{\expect[(m+n-1-\sM)_{n-1} \ind{\sM \leq m}]} \to 1, \label{eq:con100}
\end{equation}
where
\begin{itemize}
    \item if $\bar\kappa>1$ and $m \gg n \gg 1$, then $\zeta=1/(\bar\kappa-1)$;
    \item if $n/m \to \bar n>0$, then $\zeta$ is the positive root of $\zeta^2 \bar n + \zeta(\bar n + \bar\kappa -1) - 1$.
\end{itemize}
If $\kappa/m \to \bar\kappa < 1$ and $m \gg n$, as $m\to\infty$, then, for any $\epsilon>0$, as $m\to\infty$,
\begin{equation}
\frac{\expect[(m+n-1-\sM)_{n-1} \ind{\sM \leq m} \ind{ |\sM - \kappa| \leq \epsilon m} ]}{\expect[(m+n-1-\sM)_{n-1} \ind{\sM \leq m}]} \to 1. \label{eq:con200}
\end{equation}
\end{lemma}

\begin{proof}
    See Section~\ref{sec:mainproof}.
\end{proof}

The first main result characterizes the maximum of the stationary queue lengths $\mymax{L}$ in operational regimes when the marginal distribution of the queue lengths is order-$1$ and discrete~\cite{JKM22}. In such regimes, the maximum $\mymax{L}$ scales logarithmically with the number of queues, it takes discrete values, and its variance is order-$1$. For this reason, our result is stated in terms of convergence in probability rather than the (continuous) Gumbel distribution. 

\begin{theorem} \label{thm:disc} Let $\kappa/m \to \bar\kappa \geq 0$, as $m\to\infty$. Then, as $n\to\infty$,
\[
    \frac{\ln \eta}{\ln n} \mymax{L} \stackrel{\Pr}{\rightarrow} 1,
\]
where
\begin{itemize}
    \item if $\bar\kappa > 1$ and $m \gg n \gg 1$, then $\eta=\bar\kappa$;   
    \item if $n/m \to \bar n >0$, then $\eta>1$ is the root of $\eta^2 - \eta(\bar n +\bar\kappa+1) + \bar\kappa$. 
    \end{itemize}
\end{theorem}

\begin{remark}
    Under the conditions of Theorem~\ref{thm:disc}, $\Pr[L_1 \geq k] \to \eta^{-k}$, $k=0,1,\ldots$ (see~\cite{JKM22}), where~$\eta$ is as in the statement of the theorem. Hence, the asymptotic behavior of $\mymax{L}$ on the $\ln n$ scale is the same as the maximum of~$n$ independent geometric random variables with parameter $\eta^{-1}$. Recall that the geometric distribution is not in any maximum domain of attraction~\cite[p.~35]{de2010extreme}.
\end{remark}

\begin{proof} The proof is based on Lemma~\ref{lemma:concentration}. Note that $\zeta=(\eta-1)^{-1}>0$ is a root of $\zeta^2 \bar n + \zeta (\bar n + \bar\kappa -1) -1$ when $\eta>1$ is a root of $\eta^2 - \eta(\bar n +\bar\kappa+1) + \bar\kappa$. Define 
\[
x_\delta := (1+\delta) \frac{\ln n}{\ln\eta} \qquad \text{and} \qquad y_\epsilon := \frac{n}{(1-\epsilon)\eta-1}. 
\]
The representation~\eqref{eq:cdfprobrep} and conditioning on whether $\{m-\sM \leq y_\epsilon\}$ or $\{m-\sM > y_\epsilon\}$ yield, for $\delta>0$ and $\epsilon>0$,
    \begin{align*}
        \Pr[\mymax{L} > x_\delta] &= \frac{\expect[\sF_n(m-\sM) \ind{\mymax{X}(m-\sM) > x_\delta}]}{\expect[\sF_n(m-\sM)]} \\
        &\leq \frac{\expect[\sF_n(m-\sM) \ind{\mymax{X}(m-\sM) > x_\delta} \, \ind{m-\sM \leq y_\epsilon}]}{\expect[\sF_n(m-\sM)]} + \frac{\expect[\sF_n(m-\sM) \ind{\mymax{X}(m-\sM) > x_\delta} \, \ind{m-\sM>y_\epsilon}]}{\expect[\sF_n(m-\sM)]} \\
        &\leq \Pr[\mymax{X}(\ceil{y_\epsilon})>x_\delta] + \frac{\expect[\sF_n(m-\sM) \ind{m-\sM>y_\epsilon}]}{\expect[\sF_n(m-\sM)]} \\
        &\to 0,
    \end{align*}
as $n\to\infty$, for all $\epsilon$ small enough, due to Proposition~\ref{pr:max_simplex_discrete} and Lemma~\ref{lemma:concentration}. Similarly, for $\delta>0$ and $\epsilon>0$, we develop a lower bound:
    \begin{align*}
        \Pr[\mymax{L} > x_{-\delta}] &= \frac{\expect[\sF_n(m-\sM) \ind{\mymax{X}(m-\sM) > x_{-\delta}}]}{\expect[\sF_n(m-\sM)]} \\
        &\geq \frac{\expect[\sF_n(m-\sM) \ind{\mymax{X}(m-\sM) > x_{-\delta}} \, \ind{m-\sM \geq y_{-\epsilon}}]}{\expect[\sF_n(m-\sM)]} \\
        &\geq \Pr[\mymax{X}(\floor{y_{-\epsilon}}) > x_{-\delta} ] \frac{\expect[\sF_n(m-\sM) \ind{m-\sM \geq y_{-\epsilon}}]}{\expect[\sF_n(m-\sM)]} \\
        &\to 1,
    \end{align*}
as $n\to\infty$, for all $\epsilon$ small enough, due to Proposition~\ref{pr:max_simplex_discrete} and Lemma~\ref{lemma:concentration}. Combining the two bounds yields the statement of the theorem. 
\end{proof}

The following theorem covers an operational regime where individual queue lengths are order of $(m-\kappa) \gg~1$; hence, the vector~$L$ is scaled by $(m-\kappa)$. In that regime, the scaled marginal distribution is asymptotically continuous, and, as a result, $(m-\kappa)^{-1} \mymax{L}$ is asymptotically equivalent to the maximum of dependent continuous random variables, assuming the number of single-server queues~$n$ remains fixed.   

\begin{theorem} \label{thm:cont} Let $\kappa/m \to \bar\kappa <1$, as $m\to\infty$. For fixed $n \in \nat$, we have  $(m-\kappa)^{-1} L \weak X$, as $m\to\infty$, where~$X$ is uniform on a unit simplex of nonnegative reals $\Delta^n=\{x \in \real_0^n:\, |x|=1\}$. The moments of $(m-\kappa)^{-1} \mymax{L}$ converge to the moments of $\mymax{X}$. In particular, as $m\to\infty$,
\begin{align*}
\frac{n}{m-\kappa} \expect \mymax{L} &\to H_1(n), \\
\frac{n^2}{(m-\kappa)^2}\var(\mymax{L}) &\to \frac{n H_2(n) - H^2_1(n)}{n+1}.
\end{align*}
Moreover, as $n\to\infty$,
\begin{equation}
\lim_{m\to\infty}\Pr\left[\frac{n}{m-\kappa} \mymax{L} \leq x + \ln n \right] = \Pr[n\mymax{X} \leq x + \ln n] \to e^{-e^{-x}}. \label{eq:thmGumbel}
\end{equation}
\end{theorem}

\begin{remark}
\label{rm:thm:cont}
Note that $(nH_2(n)-H^2_1(n))/(n+1)$ converges to $H_2(\infty) = \pi^2/6$, as $n\to\infty$, which matches the Gumbel distribution in~\eqref{eq:thmGumbel}. However, $H^2_1(n)/n$ decays at the rate $\sim \ln^2n /n$, and hence convergences slowly. For example, the relative error for $n=100$ and $n=1000$ is approximately $21.6\%$ and $3.7\%$, respectively. 
\end{remark}

\begin{proof} We start with a proof of weak convergence. To this end, we consider two bounds for the CDF of $(m-\kappa)^{-1} L$. Define $\hat m_\epsilon:= \floor{\frac{m-\kappa}{1+\epsilon}}$.

For $\epsilon>0$, the representation~\eqref{eq:cdfprobrep} and considering whether $\{|\sM-\kappa|\leq \frac{m-\kappa}{1+\epsilon}\}$ or $\{|\sM-\kappa|>\frac{m-\kappa}{1+\epsilon}\}$ (note that $(m-\kappa) \sim (1-\bar\kappa)m$, as $m\to\infty$) yield, for $x \in \real^n$,
\begin{align}
    \Pr[L \leq (m-\kappa)x] &= \frac{\expect[\sF_n(m-\sM) \, \ind{X(m-\sM) \leq (m-\kappa)x}]}{\expect[\sF_n(m-\sM)]} \notag \\
    &\leq \Pr[X(\hat m_\epsilon) \leq (m-\kappa)x] + \frac{\expect[\sF_n(m-\sM) \, \ind{|\sM-\kappa| > \frac{\epsilon}{1+\epsilon} (m-\kappa)}]}{\expect[\sF_n(m-\sM)]}; \label{eq:proofT2-200}
\end{align}
here, $\{|\sM-\kappa| \leq \frac{\epsilon}{1+\epsilon} (m-\kappa)\}$ implies $\{m-\sM \geq \frac{m-\kappa}{1+\epsilon}\}$. The first term is estimated by exploiting the fact that $X(\hat m_\epsilon)$ is uniform on a simplex of size $\hat m_\epsilon$:
\begin{align*}
    \Pr[X(\hat m_\epsilon) \leq \hat l] &= \sum_{i \in \integer^n_0: \, i \leq l, \, |i|=\hat m_\epsilon} \frac{1}{\binom{n+\hat m_\epsilon-1}{n-1}} \\
    &= \frac{1}{\binom{n+\hat m_\epsilon-1}{n-1}}\sum_{i_1=0}^{\hat l_1} \sum_{i_2=0}^{\hat l_2 \wedge (\hat m_\epsilon-i_1)} \cdots \sum_{i_{n-1}=0}^{\hat l_{n-1} \wedge ( \hat m_\epsilon- i_{1:(n-2)})} \ind{ \hat m_\epsilon - i_{1:(n-1)} \leq \hat l_n} \\
    &\to (n-1)! \int_0^{(1+\epsilon)x_1} \int_0^{(1+\epsilon)x_2 \wedge (1-u_1)} \!\!\!\!\!\!\!\!\cdots \int_0^{(1+\epsilon)x_{n-1} \wedge (1-u_{1:(n-2)})} \!\!\!\!\!\!\ind{1 - u_{1:(n-1)} \leq x_n} \drm u_{n-1} \cdots \drm u_1,
\end{align*}
as $m\to\infty$, where $\hat l_i :=\floor{(m-\kappa) x_i}$ satisfy $\hat l_i/\hat m_\epsilon \to (1+\epsilon)x_i$, as $m\to\infty$, and $x_{1:j} := \sum_{i=1}^j x_i$. Letting $m\to \infty$ in~\eqref{eq:proofT2-200}, using Lemma~\ref{lemma:concentration} (the second term in~\eqref{eq:proofT2-200} vanishes) and the preceding limit, and then letting $\epsilon \to 0$ result in, for $x \in \Delta^n$,
\begin{align*}
\limsup_{n\to\infty} \Pr[L \leq (m-\kappa)x] &\leq (n-1)! \int_0^{x_1} \int_0^{x_2 \wedge (1-u_1)} \!\!\!\!\!\!\!\!\cdots \int_0^{x_{n-1} \wedge (1-u_{1:(n-2)})} \!\!\!\!\!\!\ind{1 - u_{1:(n-1)} \leq x_n} \drm u_{n-1} \cdots \drm u_1 \\
&= \int_{u\in\real^n_0:\, |u|=1, u\leq x} (n-1)!  \, \drm u \\
&= \Pr[X \leq x],
\end{align*}
 since $(n-1)!$ is the uniform density on a unit simplex in $\real^n_0$.

The lower bound argument is similar. For $\epsilon>0$, $\{|\sM-\kappa| \leq \frac{\epsilon}{1-\epsilon} (m-\kappa)\}$ implies $\{m-\sM \geq \frac{m-\kappa}{1-\epsilon}\}$, which, together with~\eqref{eq:cdfprobrep}, results in
\begin{align}
    \Pr[L \leq (m-\kappa)x] &= \frac{\expect[\sF_n(m-\sM) \, \ind{X(m-\sM) \leq (m-\kappa)x}]}{\expect[\sF_n(m-\sM)]} \notag \\
    &\geq \Pr[X(\hat m_{-\epsilon}) \leq (m-\kappa)x] \frac{\expect[\sF_n(m-\sM) \, \ind{|\sM-\kappa| \leq \frac{\epsilon}{1-\epsilon}(m-\kappa)}]}{\expect[\sF_n(m-\sM)]}. \label{eq:proofT2-500}
\end{align}
The first factor can be analyzed as in the previous case, which leads to, as $m\to\infty$, 
\[
\Pr[X(\hat m_{-\epsilon}) \leq \hat l] \to (n-1)! \int_0^{(1-\epsilon)x_1} \int_0^{(1-\epsilon)x_2 \wedge (1-u_1)} \!\!\!\!\!\!\!\!\cdots \int_0^{(1-\epsilon)x_{n-1} \wedge (1-u_{1:(n-2)})} \!\!\!\!\!\!\ind{1 - u_{1:(n-1)} \leq x_n} \drm u_{n-1} \cdots \drm u_1.
\]
Letting $m\to \infty$ in~\eqref{eq:proofT2-500}, using Lemma~\ref{lemma:concentration} (the fraction in~\eqref{eq:proofT2-500} converges to~$1$) and the preceding limit, and then letting $\epsilon \to 0$ result in, for $x \in \Delta^n$, 
\begin{align*}
\liminf_{n\to\infty} \Pr[L \leq (m-\kappa)x] &\geq (n-1)! \int_0^{x_1} \int_0^{x_2 \wedge (1-u_1)} \!\!\!\!\!\!\!\!\cdots \int_0^{x_{n-1} \wedge (1-u_{1:(n-2)})} \!\!\!\!\!\!\ind{1 - u_{1:(n-1)} \leq x_n} \drm u_{n-1} \cdots \drm u_1 \\
&= \Pr[X \leq x].
\end{align*}
Combining the upper and lower bounds yields the desired $(m-\kappa)^{-1} L \weak X$, as $m\to\infty$. 

The convergence of moments is implied by the weak converges due to the continuity of the $\max$ operator and $(m-\kappa)^{-1} \mymax{L} = (m-\kappa)^{-1} \, \min\{ \mymax{L}, \, m\}$. The specific limits for the first two moments follow from Proposition~\ref{pr:max_simplex_cont}. Similarly, the weak convergence and~\eqref{eq:simplexGumbel} in Proposition~\ref{pr:max_simplex_cont} yield~\eqref{eq:thmGumbel}. This completes the proof of the theorem.
\end{proof}

\section{Non-Homogeneous Network} \label{sec:nonhomogen}

In this section, we discuss briefly how our methodology can be applied to a non-homogeneous network where the utilization levels of the single-server queues differ. The results are stated in Theorems~\ref{thm:3} and~\ref{thm:4}; we omit the detailed derivations for this general case due to 
tedious technical computations.
In this regard, we first specify the model and derive a probabilistic representation of the distribution of the maximum akin to equation~\eqref{eq:maxprobrep}; see equation~\eqref{eq:maxprobrep2} below. 

Let $\kappa_i = p_i \lambda/\mu_i$, $i=1,\ldots,n$, and, without loss of generality, assume that $\kappa_i$'s are ordered:
\[
\kappa_1 = \cdots = \kappa_{n_1} < \kappa_{n_1+1} \leq \cdots \leq \kappa_n,
\]
where $n_1 \leq n$ denotes the number of queues with the highest (and identical) utilization level. Let~$\sM$ be $\kappa_1$-mean Poisson and $\{\sL_i\}_{i=n_1+1}^n$ be independent geometrically distributed random variables: 
\[
\Pr[\sL_i \geq k]= (\kappa_1/\kappa_i)^k, \qquad k=0,1,\ldots
\]
For notational convenience, let $\sL = (\sL_{n_1+1},\ldots,\sL_n) \in \integer_0^{n-n_1}$. Then, the generalization of the stationary distribution~\eqref{eq:prodpi} can be expressed in terms of $\sM$ and $\sL$:
\begin{align*}
\pi(l) &= c_{m,n}^{-1} \frac{m!}{(m- |l|)!} \prod_{i=1}^n \kappa_i^{-l_i} \\
&= c_{m,n}^{-1} \frac{m!}{(m- |l|)!} \kappa_1^{-|l|} \prod_{i=n_1+1}^n \left(\frac{\kappa_1}{\kappa_i}\right)^{l_i} \\
&= c_{m,n}^{-1} \frac{\Pr[\sM=m-|l|]}{\Pr[\sM=m] \, \Pr[|\sL|=0]} \prod_{i=n_1+1}^n \Pr[\sL_i=l_i],
\end{align*}
where $c^{-1}_{m,n}= \pi(0)$ is a normalization constant. From this product-form solution and the expression for $c_{m,n}$ in~\cite{JKM22}, the joint distribution of single-server queue lengths follows:
\begin{align*}
    \Pr[L \leq l] &= c_{m,n}^{-1} \frac{1}{\Pr[\sM=m] \, \Pr[\sL=0]} \sum_{j=0}^m \sum_{y \in \integer_0^{n-n_1}:\, |y|=j} \sum_{k=0}^{m-j} \sum_{x \in \integer^{n_1}_0: \, |x|=k} \Pr[\sM = m-j-k, \, \sL=y] \, \ind{(x, y) \leq l} \notag \\
    &= \frac{\expect[(m+n-1-\sM-|\sL|)_{n_1-1} \ind{\sM + |\sL| \leq m} \ind{(X(m-\sM -|\sL|), \, \sL) \leq l}]}{\expect[(m+n-1-\sM-|\sL|)_{n_1-1} \ind{\sM +|\sL|\leq m}]}, 
\end{align*}
where, conditional on $\{\sM+|\sL| = m-k\}$, $X(k) \in \integer_0^{n_1}$ in uniformly distributed on $\Delta^{n_1}_k$. Therefore, the distribution of the maximum can be characterized by the following ratio of expectations:
\begin{equation}
\Pr[\mymax{L} > k] = \frac{\expect[(m+n-1-\sM-|\sL|)_{n_1-1} \ind{\sM + |\sL| \leq m} \ind{\max\{\mymax{X}(m-\sM -|\sL|), \, \mymax{\sL}\} >k }]}{\expect[(m+n-1-\sM-|\sL|)_{n_1-1} \ind{\sM +|\sL|\leq m}]}. \label{eq:maxprobrep2}
\end{equation}
The structure of~\eqref{eq:maxprobrep2} mirrors that of~\eqref{eq:maxprobrep}, which suggests a two-step analysis, similar to the homogeneous case (without the $\sL$). First, one estimates the values of $\sM$ and $\sL$ on which the mass in the denominator of~\eqref{eq:maxprobrep2} is concentrated (as in Lemma~\ref{lemma:concentration}). In the operational regimes covered by Theorem~\ref{thm:disc}, these estimates are expressed in terms of large-deviation rate function of $\sM+|\sL|$, as discussed in~\cite{JKM22}. Once these estimates are obtained, the analysis proceeds with examining $\mymax{X}(m-\sM-|\sL|)$ and~$\mymax{\sL}$. Notably, when the number of utilization levels is fixed, the vector of the corresponding identically distributed $\sL_i$'s, conditional on the sum of these random variables, is uniformly distributed on a simplex as well (since the $\sL_i$'s are geometric). Thus, compared to the homogeneous case, this yields a fixed number of simplexes, each with its own maximum; the overall maximum is the maximum among these.

In order to state an analogue of Theorem~\ref{thm:disc} for a non-homogeneous network, we define the limiting logarithmic moment generating function of $\sM+|\sL|$, $\Lambda_{\sM+|\sL|}:\, [0,\infty) \to \real$, as  
\[
\Lambda_{\sM+|\sL|}(\theta) := \lim_{m\to\infty} \frac1m \ln \expect e^{-\theta (\sM+|\sL|)},
\]
when the limit exists. The logarithmic MGF is relevant here because of its Legendre-Fenchel transform, known as the rate function $\ell_{\sM+|\sL|}: [0,\infty) \to [0,\infty)$ defined by
\[
\ell_{\sM+|\sL|}(x) := \sup_{\theta \geq 0} \{-\theta x - \Lambda_{\sM+|\sL|}(\theta)\},
\]
can be used to estimate the probability of $\{\sM+|\sL| \leq xm\}$ via Cram\'er theorem
\[
\ell_{\sM+|\sL|}(x) = - \lim_{m\to\infty} \frac1m \ln \Pr[\sM+|\sL| \leq xm], \qquad x \geq 0
\]
and, consequently, the ratio of expectations in~\eqref{eq:maxprobrep2} when the network is in the regimes covered by Theorem~\ref{thm:disc}.

As in Theorem~\ref{thm:disc}, the customer pool is the bottleneck resource in the following theorem. All single-server queues are geometrically distributed, and the parameters are determined by large-deviation behavior of $\sM+|\sL|$~\cite{JKM22}. 

\begin{theorem} \label{thm:3} Suppose $\kappa_1/\kappa_i \in \{1=\varrho_1 > \varrho_2 > \cdots >\varrho_k\}$, $i=1,\ldots,n$, where $k<\infty$ and $\varrho_j$ do not vary with~$m$. Let $m^{-1} \expect[\sM+|\sL|] \to \bar\kappa_1 + \bar\sl\geq 0$, as $m\to\infty$. 
\begin{itemize}
    \item If $\bar\kappa_1 + \bar\sl> 1$ and $m \gg n_1 \gg 1$, then 
    \[
    \min_{j=1,\ldots,k} \frac{\ln \eta - \ln\varrho_j}{\ln \{\#i: \, \kappa_1/\kappa_i=\varrho_j\}} \mymax{L} \stackrel{\Pr}{\rightarrow} 1,
\]
    as $n_1\to\infty$, where $\eta$ satisfies 
    $\Lambda'_{\sM+|\sL|}(\ln \eta)=-1$;
    \item If $n_1/m \to \bar n_1 >0$, then 
    \[
    \frac{\ln \eta}{\ln n_1} \mymax{L} \stackrel{\Pr}{\rightarrow} 1,
\]
    as $n_1\to\infty$, where $\eta$ satisfies 
    $\Lambda'_{\sM+|\sL|}(\ln \eta) - \bar n_1(\eta-1)^{-1} = -1$.
    \end{itemize}
\end{theorem}

\begin{remark}
    When the network is homogeneous, the conditions of Theorem~\ref{thm:3} are reduced to the corresponding ones in Theorem~\ref{thm:disc}. In that case $n_1=n$, $|\sL|=0$ and $\bar\sl=0$, which implies $\Lambda'_{\sM+|\sL|}(\ln \eta) = \Lambda'_\sM(\ln \eta) = -\bar\kappa_1/\eta$, where $\kappa_1/m \to \bar\kappa_1$, as $m\to\infty$.
\end{remark}

The following theorem generalizes Theorem~\ref{thm:cont}. In the operational regime it covers, unlike in Theorem~\ref{thm:3}, the number of customers is sufficient to ``clog" the bottleneck queues, the $n_1$ queues with the largest relative utilization (indexed by $1,2,\ldots,n_1$). These bottleneck queues determine the overall maximum as they are proportional to $(m-\expect\sM-\expect|\sL|)/n_1$; the condition $\bar\sl<1$ ensures that the maximum of non-bottleneck queues is at most order-$\log m$. 

\begin{theorem} \label{thm:4}
    Suppose that $\kappa_1/{\kappa_{n_1+1}}< 1$ does not vary with~$m$.
    Let $m^{-1} \expect[\sM +|\sL|] \to \bar\kappa_1 + \bar\sl<1$, as $m\to\infty$. For fixed $n_1 \in \nat$, we have  $(m-\expect\sM-\expect|\sL|)^{-1} (L_1,\ldots,L_{n_1} )\weak X$, as $m\to\infty$, where~$X$ is uniform on a unit simplex of nonnegative reals $\Delta^{n_1}=\{x \in \real_0^{n_1}:\, |x|=1\}$. The moments of $(m-\expect \sM-\expect|\sL|)^{-1} \mymax{L}$ converge to the moments of $\mymax{X}$. In particular, as $m\to\infty$,
\begin{align*}
\frac{n_1}{m-\expect\sM-\expect|\sL|} \expect \mymax{L} &\to H_1(n_1), \\
\frac{n_1^2}{(m-\expect \sM - \expect|\sL|)^2}\var(\mymax{L}) &\to \frac{n_1 H_2(n_1) - H^2_1(n_1)}{n_1+1}.
\end{align*}
Moreover, as $n_1\to\infty$,
\begin{equation*}
\lim_{m\to\infty}\Pr\left[\frac{n_1}{m-\expect\sM-\expect|\sL|} \mymax{L} \leq x + \ln n_1 \right] = \Pr[n_1\mymax{X} \leq x + \ln n_1] \to e^{-e^{-x}}. 
\end{equation*}
\end{theorem}

\section{Concluding Remarks}
\label{sec:cr}


Our study of a closed network aims to develop an explicit approximation for the maximum queue distribution, thereby bypassing the difficulty of computing a normalization constant, and to uncover the global structure of the corresponding product-form solution. This global structure is obscured by the combinatorial nature of the product-form solution, despite its compact form. To circumvent this difficulty, we employ an equivalent probabilistic representation of the product-form solution, which enables a dichotomy of operating regimes and an explicit characterization of the maximum queue distribution. 


To this end, Theorem~\ref{thm:disc} (Theorem~\ref{thm:3}) addresses the operating regime where the number of queues (with the highest utilization) is large and the customer pool is the bottleneck resource. In this regime, the maximum behaves asymptotically as if the queues (with the highest utilization) are independent geometric, but with a parameter that depends on the global characteristics of the network. In contrast, Theorem~\ref{thm:cont} (Theorem~\ref{thm:4}) tackles the setting in which the number of queues (with the highest utilization) is finite and the dependencies among these queues influence the distribution of the maximum. Our analysis shows that the joint distribution of the scaled queues (with the highest utilization) is asymptotically uniform over a continuous simplex, providing a precise characterization of the maximum.



The probabilistic representation~\eqref{eq:cdfprobrep}, which relates the product-form solution to a uniform distribution on a simplex of random dimension, provides a foundation for the study of global performance measures beyond the maximum. For example, under the regimes described in Theorem~\ref{thm:disc}, one can characterize percentiles that scale in the same order as the maximum (i.e., $\ln m$). Similarly, under the conditions of Theorem~\ref{thm:cont}, where the number of queues remains fixed, the relevant characteristics of their scaled joint distribution are asymptotically equivalent to those of the uniform distribution on a continuous simplex.

Finally, we note that the regime in Theorem~\ref{thm:cont} corresponds to the situation that arises in single-cell organisms, where the number of single-server queues (filaments) is relatively small and much smaller than the number of monomers. 
Under these conditions, the distribution of the maximum---particularly its mean and variance--approximately follows the Gumbel distribution for large $n$, consistent with the approximation used in~\cite{RMJ23,MRR24}. However, Theorem~\ref{thm:cont} also reveals that the convergence of the variance to the limiting Gumbel value is rather slow, as it occurs at a rate of order $O(H_1^2(n)/n)=O(\ln^2 n/n)$; see Remark~\ref{rm:thm:cont}.

\section{Proofs} \label{sec:proofs}

\subsection{Preliminary Results} \label{sec:prelimproof}

This section contains proofs of Lemma~\ref{lemma:sum}, Proposition~\ref{pr:max_simplex_discrete} and Proposition~\ref{pr:max_simplex_cont}.

\begin{proof}[Proof of Lemma~\ref{lemma:sum}]
Let $m:= n \expect Y_1$ so that $p=m/(n+m)$ and $\sigma^2 = m(n+m)/n^2$. Because $Y_i$'s are independent geometric, their sum has negative binomial distribution: 
    \begin{align}
        \Pr[|Y| = \floor{m} +k] & = \binom{\floor{m}+k+n-1}{n-1} (1-p)^n p^{\floor{m} +k} \notag \\
        &=  \frac{n^n m^{\floor{m}+k}}{(m+n)^{\floor{m}+k+n}} \frac{n}{\floor{m}+k+n} \frac{(\floor{m}+n)!}{\floor{m}! \, n!}  \left( \prod_{j=1}^k \frac{\floor{m}+j+n}{\floor{m}+j} \right)^{\sgn(k)} \notag \\
        &= \frac{n}{\floor{m}+k+n} \frac{n^n m^{\floor{m}}}{n! \, \floor{m}!} \frac{(\floor{m}+n)!}{(m+n)^{\floor{m}+n}} \left( \prod_{j=1}^k \frac{m(\floor{m}+j+n)}{(\floor{m}+j)(m+n)} \right)^{\sgn(k)}. \label{eq:proofDS100}
    \end{align}
The product~\eqref{eq:proofDS100} is not smaller than the unity, and hence $\ln(1+x) \leq x$ yields
\begin{align}
    \ln \prod_{j=1}^{|k|} \frac{m(\floor{m}+j+n)}{(\floor{m}+j)(m+n)} &= \sum_{j=1}^{|k|} \ln\left(1 + \frac{n(\floor{m}-m+j)}{(\floor{m}+j)(m+n)} \right) \notag \\
    &\leq \sum_{j=1}^{|k|} \frac{n(\floor{m}-m+j)}{(\floor{m}+j)(m+n)} \notag \\
    &\leq \frac{|k|(|k|+1)}{2n \sigma^2} \notag \\
    &\to 0, \label{eq:proofDS200}
\end{align}
as $\min\{n,m\} \to \infty$, for all $|k| \ll \sqrt{n}\sigma$. The terms with factorials in~\eqref{eq:proofDS100} can be estimated using the Sterling approximation:
\begin{align*}
     \frac{n^n m^{\floor{m}}}{n! \, \floor{m}!} \frac{(\floor{m}+n)!}{(m+n)^{\floor{m}+n}} &\sim \frac{\sqrt{\floor{m} +n}}{\sqrt{2\pi n \floor{m}}} e^{(\floor{m}+n) \ln \frac{\floor{m} + n}{m+n} - \floor{m} \ln \frac{\floor{m}}{m}} \\
     &\sim \sqrt{\frac{m+n}{2\pi n m}},
\end{align*}
as $\min\{n,m\} \to \infty$. Finally, combining the preceding equation with~\eqref{eq:proofDS100} and~\eqref{eq:proofDS200} results in
\[
\Pr[|Y| = \floor{m} +k ] \sim \sqrt{\frac{n}{2\pi(m+n) m}} = \frac{1}{\sqrt{2\pi n \sigma^2}},
\]
as $\min\{n,m\} \to \infty$, for all $|k| \ll \sqrt{n} \sigma$.
\end{proof}

\begin{proof}[Proof of Proposition~\ref{pr:max_simplex_discrete}]
Note that for any $l \in \Delta^n_m$ one has $\Pr[X = l]=\Pr[Y = l \,| \, |Y|=m]$, where $Y=(Y_1,\ldots, Y_n)$ is a vector with independent geometrically distributed components whose parameter can be chosen to be $p=m/(n+m)$; in that case, $\var(Y_1) = \sigma^2 = m(m+n)/n^2$. The condition $\ln n \gg \ln(1+n/m)$ implies that $m \to\infty$ when $n \to\infty$, i.e., $n\to\infty$ implies $\min\{n,m\} \to \infty$. For notational simplicity let
\[
x_\epsilon : = (1+\epsilon) \frac{\ln n}{\ln (1+n/m)} = (1+\epsilon) \frac{\ln n}{-\ln p}.
\]
and 
\begin{align*}
y_\epsilon &:= (1/2+\epsilon) \frac{\ln(m+n) + \ln m + \ln n}{\ln(1+n/m)} \\
&=  (1/2+\epsilon) \frac{\ln n}{- \ln p }  \frac{\ln(m+n) + \ln m + \ln n}{\ln n} \\
&\geq x_\epsilon,
\end{align*}
where the inequality holds for $\epsilon>0$.

An upper bound can be derived from the union bound, for $\epsilon>0$:
\begin{align}
    \Pr[\mymax{X} > x_\epsilon] & = \Pr\left[\mymax{Y} > x_\epsilon \, \big| \, |Y| = m\right] \notag \\
    &\leq n \Pr[Y_1 > x_\epsilon \, | \, |Y| = m] \notag\\
    &\leq n \left( \sum_{k = \ceil{x_\epsilon}}^{\ceil{y_\epsilon}} \Pr[Y_1 = k \, | \, |Y| = m] + \Pr[Y_1 > \ceil{y_\epsilon} \, | \, |Y| = m] \right) \notag \\
    &\leq n \left( \sum_{k = \ceil{x_\epsilon}}^{\ceil{y_\epsilon}} \frac{\Pr[Y_1 = k] \, \Pr[|Y| - Y_1 = m-k]}{\Pr[|Y| = m]} + \frac{\Pr[Y_1 > \ceil{y_\epsilon} ]}{\Pr[|Y| = m]} \right) \notag \\
    &\leq n \Pr[Y_1 \geq \ceil{x_\epsilon}] \max_{\ceil{x_\epsilon} \leq k \leq \ceil{y_\epsilon}}  \frac{\Pr[|Y| - Y_1 = m-k]}{\Pr[|Y| = m]} + n \frac{\Pr[Y_1 > \ceil{y_\epsilon} ]}{\Pr[|Y| = m]}, \label{eq:proofdisc200}
\end{align}
where the third inequality is due to the independence of $\{Y_i\}_{i=1}^n$. Due to the choice of $x_\epsilon$ and $y_\epsilon$ as well as Lemma~\ref{lemma:sum} the first term vanishes: $n \Pr[Y_1 \geq \ceil{x_\epsilon}] \leq n e^{x_\epsilon \ln p} = n^{-\epsilon}\to 0$ and the maximum tends to $1$, both as $n\to\infty$. Similarly, the choice of~$y_\epsilon$ and Lemma~\ref{lemma:sum} imply that the second term also vanishes, as $n\to\infty$:
\[
n \frac{\Pr[Y_1 > \ceil{y_\epsilon} ]}{\Pr[|Y| = m]} \leq \frac{n e^{{y_\epsilon} \ln p} }{\Pr[|Y| = m]} = \frac{n (mn(m+n))^{-1/2-\epsilon} }{\Pr[|Y| = m]} \to 0;
\]
here, $\Pr[|Y|=m]$ is of the order $n^{-1/2} \sigma^{-1} = n^{1/2} (m(m+n))^{-1/2}$. Hence, $\Pr[\mymax{X} > x_\epsilon] \to 0$, as $n\to\infty$, for any $\epsilon>0$.

In order to obtain a lower bound, we introduce, for $\epsilon>0$,
\begin{equation}
N:= \sum_{i=1}^n \ind{X_i > x_{-\epsilon}}, \label{eq:Ndef}
\end{equation}
and observe that $\{N=0\}$ is equivalent to $\{\mymax{X} \leq x_{-\epsilon}\}$, $\epsilon>0$. This and Markov inequality yield
\begin{align*}
    \Pr\left[\mymax{X} > x_{-\epsilon}\right] &= 1 - \Pr[N=0] \\
    &= 1 - \Pr[\expect N - N = \expect N] \\
    &\geq 1- \Pr[|\expect N - N| \geq \expect N] \\
    &\geq 2 - \frac{\expect N^2}{(\expect N)^2}. 
\end{align*}
The second moment of~$N$ can be expressed in terms of the first moment from~\eqref{eq:Ndef}:
\begin{align*}
    \expect N^2 &= n \Pr[X_1 > x_{-\epsilon}] + n(n-1) \Pr[X_1 > x_{-\epsilon}, \, X_2 > x_{-\epsilon}] \\
    &= \expect N + n(n-1) \Pr[Y_1 > x_{-\epsilon}, \, Y_2 > x_{-\epsilon} \, |\, |Y|=m],
\end{align*}
which implies a lower bound
\begin{equation}
    \Pr\left[\mymax{X} > x_{-\epsilon}\right] \geq 2 - \frac{1}{\expect N} - \frac{n-1}{n} \frac{\Pr[Y_1 > x_{-\epsilon}, \, Y_2 > x_{-\epsilon} \, |\, |Y|=m]}{(\Pr[Y_1 > x_{-\epsilon}\, |\, |Y|=m])^2}. \label{eq:proofdisc500}
\end{equation}
The first moment of $N$ can estimated as follows:
\begin{align}
\expect N &= n \Pr[X_1> x_{-\epsilon}] \notag \\
&= n \, \Pr[Y_1 > x_{-\epsilon} \, | \, |Y| = m] \notag \\
&\geq n \sum_{k = \floor{x_{-\epsilon}} +1 }^{\ceil{y_\epsilon}} \frac{\Pr[Y_1=k] \, \Pr[|Y|-Y_1 = m-k]}{\Pr[|Y|=m]} \notag \\
&\geq n (\Pr[Y_1 > x_{-\epsilon}] - \Pr[Y_1 > \ceil{y_\epsilon}]) \, \min_{\floor{x_{-\epsilon}} \leq k \leq \ceil{y_\epsilon}} \frac{\, \Pr[|Y|-Y_1 = m-k]}{\Pr[|Y|=m]}. \label{eq:proofdisc400}
\end{align}
Due to the choice of $x_{-\epsilon}$ and $y_\epsilon$, one has 
\begin{align*}
n (\Pr[Y_1 > x_{-\epsilon}] - \Pr[Y_1 > \ceil{y_\epsilon}]) &\geq n(e^{x_{-\epsilon} \ln p} - e^{y_\epsilon \ln p}) \\
&= n \left(n^{-1+\epsilon} - (mn(m+n))^{-1/2-\epsilon}\right) \\
&\to \infty, 
\end{align*}
as $n\to\infty$, while Lemma~\ref{lemma:sum} implies that the minimum converges to~$1$, as $n\to\infty$. Hence, as $n\to\infty$,
\begin{equation}
    \expect N \to \infty. \label{eq:expectNtoinfty}
\end{equation}
For the last term in~\eqref{eq:proofdisc500}, we estimate the nominator and the denominator separately. For the former, a similar argument used to obtain~\eqref{eq:proofdisc200} yields 
\begin{align}
    &\frac{\Pr[Y_1 > x_{-\epsilon}, \, Y_2 > x_{-\epsilon} \, |\, |Y|=m]}{(\Pr[Y_1>x_{-\epsilon}])^2} \notag \\
    & \qquad \qquad \leq \sum_{k_1,k_2 = \ceil{x_{-\epsilon}} }^{\ceil{y_\epsilon}} \frac{\Pr[Y_1 =k_1, \, Y_2=k_2, \,|Y|-Y_1=m-k_1-k_2]}{\Pr[|Y|=m] \, (\Pr[Y_1>x_{-\epsilon}])^2} + \frac{1}{\Pr[|Y|=m]} \left(\frac{\Pr[Y_1> \ceil{y_\epsilon}]}{\Pr[Y_1>x_{-\epsilon}]} \right)^2 \notag \notag \\
    & \qquad \qquad \leq \max_{ 2 \ceil{x_{-\epsilon}} \leq k \leq 2 \ceil{y_\epsilon}} \frac{\Pr[|Y|-Y_1-Y_2=m-k]}{\Pr[|Y|=m]} + \frac{1}{\Pr[|Y|=m]} \left(\frac{\Pr[Y_1> \ceil{y_\epsilon}]}{\Pr[Y_1>x_{-\epsilon}]} \right)^2 \notag \\
    & \qquad \qquad \to 1, \label{eq:proofdisc700}
\end{align}
as $n\to\infty$, where the limit follows from Lemma~\ref{lemma:sum} and the choices of $x_{-\epsilon}$ and $y_\epsilon$. On the other hand, the denominator can be bounded as in~\eqref{eq:proofdisc400}:
\begin{align}
    \frac{\Pr[Y_1>x_{-\epsilon} \, | \, |Y|=m]}{\Pr[Y_1>x_{-\epsilon}]} &\geq \left(1 - \frac{\Pr[Y_1> \ceil{y_\epsilon}]}{\Pr[Y_1>x_{-\epsilon}]} \right) \min_{\ceil{x_{-\epsilon}} \leq  k \leq \ceil{y_\epsilon}} \frac{\Pr[|Y|-Y_1=m-k]}{\Pr[|Y|=m]} \notag \\
    &\to 1, \label{eq:proofdisc800}
\end{align}
as $n\to\infty$. Finally, combining~\eqref{eq:proofdisc500}, \eqref{eq:expectNtoinfty}, \eqref{eq:proofdisc700} and~\eqref{eq:proofdisc800} yields $\Pr[\mymax{X} > x_{-\epsilon}] \to 1$, as $n\to\infty$, for any $\epsilon>0$. This completes the proof of the proposition.    
\end{proof}

\begin{proof}[Proof of Proposition~\ref{pr:max_simplex_cont}]
Let $Z=(Z_1,\ldots, Z_n) \in \real^n_0$ be a vector with independent (unit-rate, without loss of generality) exponentially distributed components. We exploit the representation $X \stackrel{d}{=} Z/|Z|$ to obtain the first two moments of~$\mymax{X}$. In particular, the memoryless property implies
\begin{equation}
\mymax{X} \stackrel{d}{=} \sum_{i=1}^n \frac{Z_i/i}{|Z|}. \label{eq:maxXrep}
\end{equation}
The first moment of $\mymax{X}$ follows from~\eqref{eq:maxXrep} and symmetry:
\[
\expect \mymax{X} = \expect \left[\frac{Z_1}{|Z|}\right]  \sum_{i=1}^n \frac{1}{i}  = \frac1n H_1(n).
\]

For the second moment, a similar argument results in 
\begin{align}
    \expect \mymax{X}^2 &= \expect \left[ \sum_{i=1}^n \frac{Z_i^2}{i^2 |Z|^2} \right] + \expect \left[ \sum_{i=1}^n \sum_{j\not= i} \frac{Z_i Z_j}{ij |Z|^2} \right] \notag \\
    &= \expect \left[ \frac{Z_1^2}{|Z|^2} \right]\, H_2(n) + \expect \left[ \frac{Z_1Z_2}{|Z|^2} \right] (H_1^2(n) - H_2(n)) \notag \\
    &= \expect \left[\frac{Z_1^2}{|Z|^2} \right] (2H_2(n)- H_1^2(n)) + \frac12 \expect \left[\frac{(Z_1+Z_2)^2}{|Z|^2}\right] (H_1^2(n) - H_2(n)), \label{eq:maxX2}
\end{align}
where the last equality is based on $2Z_1Z_2 = (Z_1+Z_2)^2-Z_1^2-Z_2^2$ and symmetry. Let $Z_{-k} = Z- (Z_1+\cdots+Z_k)$. Then the two terms in the preceding equation can be evaluated as follows:
\begin{align}
    \expect \left[\frac{Z_1^2}{|Z|^2}\right] &= \int_0^1 \Pr[Z_1 > \sqrt{x} |Z|] \, \drm x \notag \\
    &= \int_0^1 \expect e^{-\frac{\sqrt{x}}{1-\sqrt{x}}|Z_{-1}|} \, \drm x \notag \\
    &= \int_0^1 \left( \expect e^{-\frac{\sqrt{x}}{1-\sqrt{x}}|Z_1|} \right)^{n-1} \drm x \notag \\
    &= \int_0^1 \left( 1-\sqrt{x} \right)^{n-1} \drm x \notag \\
    &=\frac{2}{n(n+1)}, \label{eq:Z1Z}
\end{align}
and
\begin{align}
   \expect\left[\frac{(Z_1+Z_2)^2}{|Z|^2} \right] &= \int_0^1 \Pr[Z_1+Z_2 > \sqrt{x} |Z_{-2}|] \, \drm x \notag \\
   &= \int_0^1 \left(\expect e^{-\frac{\sqrt{x}}{1-\sqrt{x}}|Z_{-2}|} + \frac{\sqrt{x}}{1-\sqrt{x}}\expect\left[|Z_{-2}| e^{-\frac{\sqrt{x}}{1-\sqrt{x}}|Z_{-2}|}\right] \right) \drm x \notag \\
   &= \frac{2}{(n-1)n} + \int_0^1 \frac{\sqrt{x}}{1-\sqrt{x}} \frac{\drm}{\drm \frac{\sqrt{x}}{1-\sqrt{x}}} \expect e^{-\frac{\sqrt{x}}{1-\sqrt{x}}|Z_{-2}|} \, \drm x \notag \\
   &= \frac{2}{(n-1)n} + \int_0^1 \frac{\sqrt{x}}{1-\sqrt{x}} \frac{\drm}{\drm \frac{\sqrt{x}}{1-\sqrt{x}}} \frac{1}{(1+\frac{\sqrt{x}}{1-\sqrt{x}})^{n-2}} \, \drm x \notag \\
   &= \frac{2}{(n-1)n} - (n-2) \int_0^1 \sqrt{x} (1-\sqrt{x})^{n-2} \, \drm x \notag \\
   &= \frac{2}{(n-1)n} + \frac{4(n-2)}{(n-1)n(n+1)} \notag \\
   &= \frac{6}{n(n+1)}. \label{eq:Z12Z}
\end{align}
Substituting~\eqref{eq:Z1Z} and~\eqref{eq:Z12Z} in~\eqref{eq:maxX2} yields
\[
\expect \mymax{X}^2 = \frac{H^2_1(n) + H_2(n)}{n(n+1)},
\]
and the expression for $\var(\mymax{X})$ follows.

In the last part of the proof, we show~\eqref{eq:simplexGumbel}. Here again we use $n X \stackrel{d}{=} Z/(|Z|/n)$ and utilize the CLT to argue that $|Z|/n \approx 1$. In particular, we develop two asymptotically matching bounds. For $c>0$, an upper bound follows from the exponential nature of $Z_i$'s and the CLT: 
\begin{align*}
    \Pr[n\mymax{X} \leq x +\ln n] &\leq \Pr[\mymax{Z} \leq (1+c/\sqrt{n})(x+\ln n)] + \Pr[|Z| > n + c\sqrt{n}] \\
    &= \left(1 -  \frac{e^{-(1+c/\sqrt{n})}}{n^{1+c/\sqrt{n}}}\right)^n + \Pr[|Z| > n + c\sqrt{n}] \\
    &\to e^{-e^{-x}} + \Phi(-c),
\end{align*}
as $n\to\infty$, where $\Phi$ is the standard normal CDF. The lower bound is similar conceptually:
\begin{align*}
    \Pr[n\mymax{X} \leq x +\ln n] &\geq \Pr[\mymax{Z} \leq (1-c/\sqrt{n})(x+\ln x), \, |Z| > n -c\sqrt{n}] \\
    &\geq \Pr[\mymax{Z} \leq (1-c/\sqrt{n})(x+\ln x)] - 1 + \Pr[|Z| > n -c\sqrt{n}] \\
    &= \left(1 -  \frac{e^{-(1-c/\sqrt{n})}}{n^{1-c/\sqrt{n}}}\right)^n - \Pr[|Z|\leq n-c\sqrt{n}] \\
    &\to  e^{-e^{-x}} - \Phi(-c),
\end{align*}
as $n\to\infty$. Finally, by letting $c\to\infty$ in the two bounds, one derives~\eqref{eq:simplexGumbel}.    
\end{proof}

\subsection{Proof of Lemma~\ref{lemma:concentration}} \label{sec:mainproof}

\begin{proof} We start with a representation of scaled summands in $\expect[\sF_n(m-\sM)]$:
\begin{align*}
    \frac{e^n n^{-n}}{\Pr[\sM=m]} \, \sF_n(i) \, \Pr[m-\sM=i] &= e^n n^{-n} \frac{m!}{e^{-\kappa} \kappa^m} (n-1+i)_{n-1} \frac{\kappa^{m-i}}{(m-i)!} e^{-\kappa} \\
    &= e^n n^{-n} \frac{(n-1+i)! \, m!}{i! \, (m-i)!} \kappa^{-i} \\
    &= h_{m,n}(i) \, e^{-g_{m,n}(i)},
\end{align*}
where 
\begin{align}
g_{m,n}(i) &:= -(n+i) \ln(n+i) +n\ln n -m \ln m + i \ln i + (m-i) \ln(m-i) + i \ln \kappa + i \notag\\
&= -n \ln\left(1+\frac{i}{n}\right) + m\ln\left(1-\frac{i}{m}\right) + i  -i \ln\left(1+\frac{n}{i} - \frac{i}{m} - \frac{n}{m}\right) + i \ln \frac{\kappa}{m}, \label{eq:gmn}
\end{align}
and
\begin{equation}
h_{m,n}(i) := \frac{(n+i)! \, m!}{(n+i)\invexcl\, m\invexcl} \frac{i\invexcl \, (m-i)\invexcl}{i! \, (m-i)!} \sqrt{\frac{m}{(m-i)i(n+i)}}; \label{eq:hmn}
\end{equation}
here $x\invexcl := \sqrt{2\pi x} x^x e^{-x}$ is Sterling's approximation for $x!$. The function $g_{m,n}: [0,m] \to \real$ is convex, because
\[
g''_{m,n}(x) = \frac{x^2+mn}{(m-x)x(n+x)} >0,
\]
for $x\in (0,m)$, and it achieves a unique minimum at the root of $x^2 + x(n +\kappa-m)-nm$, because $g'_{m,n}(x) = -\ln(n+x) + \ln x - \ln(m-x) + \ln \kappa$. 

Now, we consider the three asymptotic regimes from the statement of the lemma:
\begin{itemize}
    \item ($\bar\kappa>1$ and $m \gg n \gg 1$) The unique minimum of $g_{m,n}(\cdot)$ is proportional to~$n$. In particular, the minimum is at $n \zeta_n \in [0,m]$, where $\zeta_n$ is the root of
\begin{equation}
\zeta_n^2 \frac{n}{m} + \zeta_n \left(\frac{n}{m}+\frac{\kappa}{m} -1 \right) - 1, \label{eq:zetaroot1}
\end{equation}
and, hence, $\zeta_n\to \zeta=1/(\bar\kappa-1)>0$, as $n\to\infty$ (recall $m \gg n$). The value of the minimum is also proportional to~$n$ (rather than~$m$):
\begin{align*}
\frac1n g_{m,n}(n\zeta_n) 
&\to - \ln \frac{\bar\kappa}{\bar\kappa-1} \\
&= g(\zeta),
\end{align*}
as $n\to\infty$, where
\[
g(x):= -\ln(1+x) - x\ln \left(1+\frac1x \right) + x \ln \bar\kappa,
\]
and $\zeta=1/(\bar\kappa-1)>0$ is the unique minimizer of $g(\cdot)$ on $[0,\infty)$. Note that, for large values of the argument, $g(\cdot)$ consists of a logarithmic term, a constant term, and a linear term. That motivates a linear bound based on~\eqref{eq:gmn}. In particular, for any $\delta>0$ there exists $c_\delta >0$ such that, for all $i > c_\delta n$,
\begin{align}
\frac{g_{m,n}(i)}{i} &= - \frac{n}{i} \ln \left(1 + \frac{i}{n} \right) - \ln\left(1 + \frac{n}{i}\right) + \left(\frac{m}{i}-1 \right) \ln \left(1 - \frac{i}{m} \right) + 1   + \ln \frac{\kappa}{m} \notag \\
&\geq -\delta + \ln\frac{\kappa}{m}, \label{eq:case1linear}
\end{align}
because the sum of the first two terms is decreasing to $0$ as $i/n \to \infty$, and the sum of the next two terms is nonnegative.

Now, we consider three limits. First, for $\epsilon>0$, we use
\[
\sum_{i:\, |i-\zeta n| \leq \epsilon n} h_{m,n}(i)\,  e^{-g_{m,n}(i)} \geq h_{m,n}(\floor{n \zeta})\,  e^{-g_{m,n}(\floor{n \zeta})}
\]
to obtain
\begin{align}
    \frac1n \ln\left(\frac{e^n n^{-n}}{\Pr[\sM=m]} \,\expect[\sF_n(m-\sM)  \ind{ |(m-\sM)/n - \zeta| \leq \epsilon} ]\right) &\geq \frac1n \ln h_{m,n}(\floor{n\zeta}) - \frac1n g_{m,n}(\floor{n\zeta}) \notag \\
    &\to - g(\zeta), \label{eq:case1eq100}
\end{align}
as $n\to\infty$, where the limit is also due to $n h_{m,n}(\floor{n \zeta}) \to (\zeta(1+\zeta))^{-1} = (1-\bar\kappa)^2/\bar\kappa$, as $n\to\infty$. Second, for $0 < \epsilon < c$, we utilize
\[
\sum_{i:\, \epsilon n < |i-\zeta n| \leq cn} h_{m,n}(i)\,  e^{-g_{m,n}(i)} \leq 2cn \, \max_{i:\, \epsilon n < |i-\zeta n| \leq cn} h_{m,n}(i) \, \max_{i:\, \epsilon n < |i-\zeta n| \leq cn} e^{-g_{m,n}(i)}
\]
to derive
\begin{align}
    \frac1n \ln &\left(\frac{e^n n^{-n}}{\Pr[\sM=m]} \,\expect[\sF_n(m-\sM)  \ind{ \epsilon < |(m-\sM)/n - \zeta| \leq c} ]\right) \notag \\
    &\leq \frac1n \ln (2cn) + \max_{i:\, \epsilon n < |i-\zeta n| \leq cn} \frac1n \ln h_{m,n}(i) - \min_{i:\, \epsilon n < |i-\zeta n| \leq cn} \frac1n  g_{m,n}(i) \notag \\
    &\to - \min_{x:\, \epsilon<|x-\zeta|\leq c} g(x)\notag \\
    &< - g(\zeta), \label{eq:case1eq200}
\end{align}
as $n\to\infty$, where the limit holds because $h_{m,n}$ is polynomial in~$n$ (see~\eqref{eq:hmn}), and the last inequality holds because $\zeta$ achieves the minimum of $g(\cdot)$. Third, for $0<\delta< \ln\bar\kappa$, select $c > \max\{c_\delta,\zeta\}$ such that $c(-\delta + \ln \bar\kappa)> g(\zeta)$. Then 
\begin{align}
    \frac1n \ln &\left(\frac{e^n n^{-n}}{\Pr[\sM=m]} \,\expect[\sF_n(m-\sM)  \ind{ |(m-\sM)/n - \zeta| > c} ]\right) \notag \\
    &\leq \frac1n \ln \left(\sum_{i: i> cn} h_{m,n}(i)\, e^{-i (-\delta + \ln \frac{\kappa}{m})}\right) \notag \\
    &\leq \frac1n \max_{i: \, i>cn} \ln h_{m,n}(i) - \frac{cn}{n} \left(-\delta +\ln \frac{\kappa}{m} \right) + \frac1n \ln \left(1- e^{\delta - \ln\frac{\kappa}{m}} \right) \notag \\
    &\to -c (-\delta + \ln \bar\kappa) \notag \\
    &< - g(\zeta), \label{eq:case1eq300}
\end{align}
as $n\to\infty$, due to the choice of $\delta$ and $c$. Finally, combining~\eqref{eq:case1eq100}, \eqref{eq:case1eq200}, and~\eqref{eq:case1eq300} results in
\[
\limsup_{n\to\infty}\frac1n \ln \frac{\expect[\sF_n(m-\sM)  \ind{ |(m-\sM)/n - \zeta| > \epsilon} ]}{\expect[\sF_n(m-\sM)  \ind{ |(m-\sM)/n - \zeta| \leq \epsilon} ]} < 0,
\]
from which~\eqref{eq:con100} follows.

    \item ($n/m \to \bar n>0$) In this regime, $g_{m,n}(\cdot)$ has a unique minimum at $n \zeta_n \in [0,m]$, where~$\zeta_n$ is the root of~\eqref{eq:zetaroot1} (as in the previous case), but $\zeta_n \to\zeta$, as $n\to\infty$, where $\zeta$ is the positive root of $\zeta^2 +\zeta(\bar n +\bar\kappa -1)-1$. The value of the minimum $g_{m,n}(n\zeta_n)$ is also order-$n$:
\begin{align*}
\frac1n g_{m,n}(n\zeta_n) 
&\to - \ln\left(1+\zeta\right) + \bar n^{-1} \ln\left(1- \bar n \zeta\right) + \zeta - \zeta \ln\left(1+ \zeta^{-1} - \bar n \zeta  - \bar n\right) + \zeta \ln \bar\kappa \\
&= - \ln(1+\zeta) + \bar n^{-1} \ln(1-\bar n \zeta) + \zeta \\
&= g(\zeta), 
\end{align*}
as $n\to\infty$, where, for $x\in[0,\bar n^{-1}]$,
\[
g(x) := -\ln(1+x) + \bar n^{-1} \ln(1 - \bar n x)   + x - x \ln\left(1+ x^{-1} - \bar n x - \bar n \right) +x \ln \bar\kappa,
\]
and $g(\cdot)$ has a unique minimizer $\zeta$ on $[0,\bar n^{-1}]$. A lower bound is obtained by considering a single term in the expectation $\expect[\sF_n(m-\sM)]$. In particular, for any $\epsilon>0$, 
\begin{align}
    \frac1n \ln\left(\frac{e^n n^{-n}}{\Pr[\sM=m]} \,\expect[\sF_n(m-\sM)  \ind{ |(m-\sM)/n - \zeta| \leq \epsilon} ]\right) &\geq \frac1n \ln h_{m,n}(\floor{n\zeta_n}) - \frac1n g_{m,n}(\floor{n\zeta_n}) \notag \\
    &\to - g(\zeta), \label{eq:case2eq100}
\end{align}
as $n\to\infty$, where the limit is due to $h_{m,n}(\floor{n\zeta_n})$ being polynomial in~$n$ (or equivalently~$m$). For an upper bound, we use the fact that the total number of summands is~$m$:
\begin{align}
    \frac1n \ln &\left(\frac{e^n n^{-n}}{\Pr[\sM=m]} \,\expect[\sF_n(m-\sM)  \ind{ |(m-\sM)/n - \zeta| > \epsilon} ]\right) \notag \\
    &\qquad \leq \frac{1}{n}\ln m + \frac{1}{n}\ln \left( \max_i  h_{m,n}(i)\right) - \min_{i:\, |i/n - \zeta|> \epsilon}  \frac1n g_{m,n}(i) \notag \\
    &\qquad \to - \min_{x:\, |x - \zeta|> \epsilon } g(x) \notag \\
    &\qquad < - g(\zeta), \label{eq:case2eq200}
\end{align}
as $n\to\infty$, for $\epsilon>0$, where the last inequality follows from $\zeta$ being the unique minimizer. Combining~\eqref{eq:case2eq100} and~\eqref{eq:case2eq200} yields
\[
\limsup_{n\to\infty}\frac1n \ln \frac{\expect[\sF_n(m-\sM)  \ind{ |(m-\sM)/n - \zeta| > \epsilon} ]}{\expect[\sF_n(m-\sM)  \ind{ |(m-\sM)/n - \zeta| \leq \epsilon} ]} < 0,
\]
which, in turn, implies~\eqref{eq:con100}.
    
    \item ($\kappa/m \to \bar\kappa < 1$ and $m \gg n$) The function $g_{m,n}(\cdot)$ achieves a unique minimum at $m \zeta_m \in [0,m]$, where $\zeta_m$ is the root of
\[
\zeta_m^2  + \zeta_m \left(\frac{n}{m}+\frac{\kappa}{m} -1 \right) - \frac{n}{m},
\]
and, hence, $\zeta_m\to 1-\bar\kappa>0$, as $m\to\infty$. The value of the minimum is also on the order-$m$ scale:
\begin{align*}
\frac1m g_{m,n}(m\zeta_m) &\to  \ln \bar\kappa + 1 - \bar\kappa \\
&= g(1-\bar\kappa),
\end{align*}
as $m\to\infty$, where, for $x\in[0,1]$,
\[
g(x) := \ln(1-x) + x + x \ln \frac{\bar\kappa}{1-x},
\]
which has a unique minimum at $1-\bar\kappa \in[0,1]$. The lower and upper bounds are similar to those from the previous case. For $\epsilon>0$, a lower bound follows by considering a single element of the sum:
\begin{align*}
    \frac1m \ln\left(\frac{e^n n^{-n}}{\Pr[\sM=m]} \,\expect[\sF_n(m-\sM)  \ind{ |\sM/m - \bar\kappa| \leq \epsilon} ] \right) 
&\geq \frac{\ln h_{m,n}(\floor{m \zeta_m})}{m} - \frac1m g_{m,n}(\floor{m \zeta_m}) \\
&\to -\ln \bar\kappa - 1 + \bar\kappa, 
\end{align*}
as $m\to\infty$; while, for an upper bound, we maximize $h_{m,n}$ and minimize $g_{m,n}$ for each element:
\begin{align*}
    \frac1m \ln &\left( \frac{e^n n^{-n}}{\Pr[\sM=m]} \,\expect[\sF_n(m-\sM)  \ind{ |\sM/m - \bar\kappa| > \epsilon} ] \right) \\
    & \qquad \leq \frac{\ln m}{m} + \frac{\ln \max_i h_{m,n}(i)}{m} - \min_{i:\, |i - (1-\bar\kappa) m|> \epsilon m}  \frac1m g_{m,n}(i) \\
    &\qquad \to - \min_{x:\, |x - 1+\bar\kappa|> \epsilon } g(x) \\
    & \qquad < -\ln \bar\kappa - 1 + \bar\kappa,
\end{align*}
as $m\to\infty$, because $1-\bar\kappa$ is the unique minimum of~$g(\cdot)$. The upper and the lower bound yield~\eqref{eq:con200}.
\end{itemize}
This completes the proof of the lemma.
\end{proof}

%

\bibliographystyle{plain}
\bibliography{petarbib}
\end{document}